\documentclass[12pt]{amsart}

\usepackage{amsfonts}
\usepackage{enumerate}

\date{}

\usepackage{graphicx}
\usepackage{amsmath,amsthm}
\usepackage{amsfonts}
\usepackage{amssymb}
\usepackage{latexsym}
\usepackage{tikz}
\usepackage{fancyvrb}
\usepackage{datetime}

\usepackage{graphicx}
\usepackage{color}
\usepackage[colorlinks]{hyperref}

\newtheorem{lemma}{Lemma}[section]

\newtheorem{corollary}{Corollary}[section]

\newtheorem{theorem}{Theorem}[section]

\begin{document}
\begin{center}

\title{\bf{No two  Jellyfish graphs are $L$-cospectral and $Q$-cospectral}}
\end{center}

\author{\bf A.Z. Abdian and  A.R. Ashrafi$^\star$ }

\address{\textbf{Ali Zeydi Abdian}, Department of Mathematical Sciences, Lorestan University, College of Science, Lorestan, Khoramabad, Iran, E-mail: azeydiabdi@gmail.com; aabdian67@gmail.com; abdian.al@fs.lu.ac.ir}

\address{\textbf{ Ali Reza Ashrafi},  Department of Pure Mathematics, Faculty of Mathematical Sciences, University of Kashan, Kashan 87317-53153, E-mail: ashrafi@kashanu.ac.ir}

\thanks{$^\star$Corresponding author (Email: ashrafi@kashanu.ac.ir)}

\maketitle

\begin{abstract}
If $q$ copies of $K_{1,p}$ and a cycle $C_q$ are joined by merging any vertex of $C_q$ to the vertex with maximum degree of $K_{1, p}$, then the resulting graph is called the jellyfish graph $JFG(p, q)$ with parameters $p$ and $q$. Two graphs are said to be $Q$-cospectral (respectively, $L$-cospectral) if they have the same signless Laplacian (respectively, Laplacian) spectrum. A graph is said to be DQS (respectively,  DLS)  if there is no other non-isomorphic graphs $Q$-cospectral (respectively, $L$-cospectral) with it. In [M. Mirzakhah and D. Kiani, The sun graph is determined by its signless Laplacian spectrum, Electron J. Linear Algebra,  20 (2010) 610--620] it were proved that the sun graphs are DQS, where $Q(G)$ is used for the signless Laplacian matrix of $G$. Additionally, in [R. Boulet, Spectral characterizations of sun graphs and broken sun graphs, Discrete Math. Theor. Comput. Sci. 11 (2) (2009) 149–160] it was proved that the sun graphs are also DLS, where $L(G)$  denotes the Laplacian matrix of $G$. In this paper, it is proved that the jellyfish graphs, a natural generalization of sun graphs,  are both  DLS (for when $q$ is an  even number) and DQS.

\vskip 3mm


\vskip 3mm

\end{abstract}

\section{Introduction}
Throughout this paper, as usual $G = (V, E)$ will denote a simple graph having $n$ vertices and $m$ edges, with $V = \{v_1, v_2, \ldots, v_n\}$ and $E =\{e_1, e_2, \ldots, e_m\}$. The complement of $G$ is denoted $\overline{G}$. Another graph operation that will be useful here is the disjoint union of $r$ copies of a graph $G$ being denoted by $rG$. Consistent with this notation, we let $G + H$ denote the disjoint union of graphs $G$ and $H$. The join $G * H$ of graphs $G$ and $H$ is obtained from $G + H$ by joining each vertex of $G$ to each vertex of $H$. Our next operation applies only to rooted graphs, that is, graphs in which one vertex is singled out as being the root: if $G$ and $H$ are rooted graphs, then their coalescence $G \bullet H$ is obtained from $G + H$ by identifying their roots.

Suppose $M$ is a function from the set of all simple  graphs into the set of all square matrices on $\mathbb{R}$ such that (i) for each graph $G$, the order of $G$ and the size of $M(G)$ are equal; and (ii) if $G \cong H$ then $M(G)$ and $M(H)$ are cospectral. Then the function $M$ is called a \textit{graph characteristic function} and the matrix $M(G)$ is called the \textit{$M$-matrix} of $G$.  Two graphs $G$ and $H$ with this property that their $M$-matrices have the same spectrum are said to be $M$-cospectral. A graph is said to be DMS if there is no other non-isomorphic graphs $M$-cospectral with it. In literature, Three natural cases of the function $M$ are studied. These are as follows:
\begin{enumerate}
\item $M(G)=A(G)$ in which $A(G)$ denotes the \textit{adjacency matrix} of $G$. The spectral graph theory originated with the study of eigenvalues of this matrix.

\item $M(G)=L(G)$, where $L(G) = A(G) - D(G)$ is the \textit{Laplacian matrix} of $G$. Here, $D(G)$ is the diagonal matrix ${\rm{Diag}}(d_1, d_2, \ldots, d_n)$ in which $d_i$ is the degree of vertex $v_i$.

\item $M(G)=Q(G)$ such that $Q(G) = A(G) + D(G)$ is the signless Laplacian matrix of $G$.
\end{enumerate}

In this paper we focus on the Laplacian matrix. Let $ \mu_1 > \mu_2 > \cdots > \mu_t$ be the distinct eigenvalues of $L(G)$ with multiplicities $m_1$, $ m_2$, $\cdots$, $m_t$, respectively. van Dam and Haemers \cite{VH} conjectured that almost all graphs are DQS or DLS. There a few classes of graphs which are known to satisfy this property, and so it is an interesting problem to find new classes of such graphs.

Suppose $ {\rm{Spec}}_{Q}(G) = \left\{ { [q_1]^{ m_1 } , [q_2]^{ m_2 } ,\cdots, [q _n]^{ m_n } } \right\}$ is the multi-set of eigenvalues of $ Q(G) $, where $ m_i $ denote the multiplicities of $q_i$. Conventionally, the signless Laplacian eigenvalues of graph $G$ are ordered respectively in non-increased sequence as follows:   $q_1\geq q_2\geq \cdots\geq q_n$.

Mirzakhah and Kiani \cite{MK} proved that the sun graphs are DQS and Boulet \cite{B} proved that the sun graphs are also DLS. The aim of this paper is to generalize these results to the jellyfish graphs. In an exact phrase, we will prove the following result:

\begin{enumerate}
\item The jellyfish graphs $G = JFG(p, q)$  are  DQS.

\item Let $H$ be any graph $L$-cospectral to a jellyfish graph $G = JFG(p, q)$. If $q$ is an even number, then $H$ and as a result its complement are DLS.
\end{enumerate}

Our notations are standard and we refer to  Cvetkovi\'c,  Rowlinson and  Simi\'c \cite{CRS} for basic definitions and results in algebraic graph theory.

\section{Preliminaries}
In this section we present some  results which are crucial throughout this paper. Suppose $G$ is a simple graph and $M_1(G) = \sum_{v \in V(G)}deg(v)^2$. The quantity $M_1(G)$ is well-studied in literature and is called the first Zagreb index of $G$, see \cite{gut1,gut2} for details.

It is well-known that the Laplacian spectrum of a graph determines the number of vertices, the number of edges, the number of spanning trees, the number of components and the first Zagreb index of $G$. We note in passing that the spectrum of the adjacency matrix of a graph gives other information, including the number of closed walks of any given length, whether the graph is bipartite or not, whether it is regular or not, and if it is, the degree of regularity.

The next theorem relates the Laplacian spectra of complementary graphs.

\begin{theorem}[\cite{K1}]\label{thm 2-2} Let $\mu_1 \geq \mu_2 \geq \ldots \geq \mu_n=0$ and $\overline{\mu} _1\geq \overline{\mu}_2\geq \ldots \geq \overline{\mu}_n=0$ be the Laplacian spectra of $G$ and $\overline{G}$, respectively. Then $\overline{\mu}_i = n - \mu_{n-i}$ for $i = 1,2,\ldots, n - 1$.
\end{theorem}

For graphs $G$ and $H$, we let $N_G(H)$ be the number of subgraphs of graph $G$ that are isomorphic to $H$. Further, let $W_G(i)$ be the number of closed walks of length $i$ in $G$ and $W'_H(i)$ be the number of closed walks of length $i$ in $H$ that cover the edges of $H$. Then $W_G(i)=\sum{{N_G(H)W'_H(i)}}$, where the sum is taken over all connected subgraphs $H$ of $G$ for which $W'_H(i)\neq 0$. This equation provides some formulas for calculating the number of some short closed walks in $G$. Note that if tr$(M)$ denotes the trace of a matrix M, then $W_G(3) = \mathrm{tr}(A^3(G))$. It is easy to see that an $n$-cycle have exactly $2n$ closed walks of length $n$.

\begin{theorem}[\cite{WH}]\label{thm 2-3} The number of closed walks of lengths $2$, $3$, and $4$ in a graph $G$ with exactly $m$ edges are as follows:
$(i)$ $W_G(2)=2m$,
$(ii)$ $W_G(3) = {\mathrm{tr}(A^3(G))} = 6N_G(C_3)$,
$(iii)$ $W_G(4)=2m+4N_G(P_3)+8N_G(C_4)$.
\end{theorem}

Turning to the degrees of the vertices in graphs, as before, we let $d_i$ denote the degree of vertex $v_i$ in a graph $G$, and assume that $d_1 \geq d_2 \geq \ldots \geq d_n$. In addition, the eigenvalues of $G$ are assumed to be $\mu_1 \geq \mu_2 \geq \ldots \geq \mu_n=0$.

\begin{theorem}[\cite{AABO, GM}]\label{thm 2-4} If $G$ is a graph with at least one edge, then $\mu_1(G) \geq d_1(G) + 1$. Moreover, if $G$ is connected, then equality holds if and only if $d_1(G) = n-1$.
\end{theorem}

A graph with exactly two different vertex degrees is said to be semi-regular. The next result uses the quantity $\theta(v) = \Sigma \dfrac{{\rm{deg}} u}{\rm{deg} v}$, where the sum is taken over the neighbors $u$ of the vertex $v$.

\begin{theorem}[\cite{WH}]\label{thm 2-5}
If $G$ is a connected graph, then $\mu_1(G) \leq \max_v(\deg(v)+ \theta(v))$. Moreover, equality holds if and
only if $G$ is a regular bipartite graph or a semi-regular bipartite graph. \end{theorem}

In the following theorem, closed formulas for the first four coefficients of the characteristic polynomial of a graph $G$ are given.

\begin{theorem}[\cite{OAJ}]\label{thm 2-6} The first four coefficients in the characteristic polynomial $\varphi(G) = \Sigma l_ix^i$ of a graph $G$ are $l_0=1$, $l_1=-2m$, $l_2=2m^2-m-\dfrac{1}{2}\sum_{i = 1}^{{n}} {{d^2_i}},$
and $l_3=\dfrac{1}{3}(-4m^3+6m^2+3m\sum_{i = 1}^{{n}} {{d^2_i}}-\sum_{i = 1}^{{n}} {{d^3_i}}-3\sum_{i = 1}^{{n}} {{d^2_i}}+6N_G(C_3))$.
\end{theorem}

Suppose $deg(G)=(d_1, \cdots, d_n)$ is the degree sequence of a graph $G$. In the following theorem some exact expressions for the first four spectral moments of the $Q$-spectrum of $G$ are given.

\begin{lemma}[\cite{CRS3, S1}]\label{lem 2-10}
Let $G$ be a graph  with $n$ vertices, $m$ edges, $N_G(C_3)$ triangles and degree sequence $deg(G)=(d_1, \cdots, d_n)$. Let $T_k=\sum_{i = 1}^{{n}} {{q^k_i}}$, $0 \leq k \leq n,$ be the k-th spectral moment for the $Q$-spectrum of $G$. Then $T_0=n$, $T_1=\sum_{i = 1}^{{n}} {{d_i}}=2m$, $T_2=2m+\sum_{i = 1}^{{n}} {{d^2_i}}$, $T_3=6N_G(C_3)+3\sum_{i = 1}^{{n}} {{d^2_i}}+\sum_{i = 1}^{{n}} {{d^3_i}}$.
\end{lemma}

Note that  $q_n(G)\geq 0$ in general.

\begin{lemma}[\cite{CRS3}]\label{lem 2-11} The multiplicity of the eigenvalue 0 in the $Q$-spectrum denotes the number of bipartite components.
\end{lemma}

A \textit{unicyclic graph} is a connected graph with this property that the number of vertices and edges are equal. Such a graph has exactly one cycle. If this cycle has an odd length then the unicyclic graph is said to be odd.

\begin{lemma}[\cite{MK}]\label{lem 2-12} Let $G$ be a graph with $n$ vertices and $m$ edges;
\begin{enumerate}
\item[$($i$)$] $det(Q(G)) = 0$ if and only if $G$ has at least one bipartite connected component.

\item[$($ii$)$] $det(Q(G)) = 4$ if and only if $G$ is an odd unicyclic graph.

\item[$($iii$)$] Suppose $u$ and $v$ are two non-adjacent vertices in the graph $G$ containing the same neighbors and $deg(u) =deg(v) =r$.  Then $r\in {\rm{Spec}}_{Q}(G)$.
\end{enumerate}
\end{lemma}

Suppose $G$ is a graph. The line graph of $G$, $\mathcal{L}(G)$, is a graph with vertex set $E(G)$ in which two edges of $G$ are adjacent if and only if they have a common vertex.

\begin{lemma}\label{lem 2-13}
The following hold:
\begin{enumerate}
\item[$($1$)$] {\rm (\cite{S11})} Let $G$ be a connected unicyclic bipartite graph with $n$ vertices and $\mathcal{L}(G)$ its line graph. Then $\mu_i(G) =\lambda_i(\mathcal{L}(G))+2$,  for $i= 1,2, . . . , n-1$, where $λi(\mathcal{L}(G))$ is the i-th largest adjacency eigenvalue of $\mathcal{L}(G)$.

\item[$($2$)$] {\rm (\cite{ZB})} If two graphs $G$ and $H$ are $Q$-cospectral, then their line graphs are
$A$-cospectral. The converse is true if $G$ and $H$ have the same number of vertices and edges.
\end{enumerate}
\end{lemma}

We end this section with the following useful result:

\begin{lemma}[\cite{CZ}]\label{lem 2-16}
Let $H$ be a proper subgraph of a connected graph $G$. Then, $q_1(G)>q_1(H)$.
\end{lemma}

\section{Proof of the Main Result}

The aim of this section is to prove that the jellyfish graphs   $G = JFG(p, q)$ are both  DLS (if $q$ is an even number) and DQS.

\begin{lemma}\label{lem 3-1}
If $H$ is a graph $L$-cospectral with $G = JFG(p, q)$, then $p+3\leq \mu_1(H)\leq p+3+\dfrac{2}{p+2}$.
\end{lemma}

\begin{proof} By Lemma \ref{thm 2-4}, $\mu_1(G)\geq p+3$ and by Lemma \ref{thm 2-5}, $\mu_1(G)\leq p+3+\dfrac{2}{p+2}$. This implies that $p+3\leq \mu_1(H)\leq p+3+\dfrac{2}{p+2}$, as desired.
\end{proof}

Let $G$ be a connected  graph with $n$ vertices and $m$ edges. If $k = m - n + 1$, then $G$ is said to be $k$-cyclic graph.  Obviously, any $k$-cyclic graph consists of $k$ cycle(s). Consider the jellyfish graph $G = JFG(p, q)$, then $n=n(G)=q(1+p)$ and $m=m(G)=q(1+p)$ and so $m=n$. This shows that the jellyfish graph $G$  is an unicyclic graph.

\begin{lemma}\label{lem 4-2}
If $H$ is a  graph $L$-cospectral with $G =  JFG(p, q)$ and $q$ is an even number, then they have the same degree sequence.
\end{lemma}

\begin{proof}
Since $H$ and $G$ are $L$-cospectral,  $H$ is also connected, and has the same order, size, and the first Zagreb index as $G$. Let $n_i$ denote the number of vertices of degree $i$ in $H$, for $i=1, 2, \ldots, d_1(H)$. Then,

\begin{eqnarray}
\sum_{i = 1}^{{d_1}(H)} {{n_i}} &=& n(G),\\
\sum_{i = 1}^{{d_1}(H)} {{in_i}} &=& 2m(G),\\
\sum_{i = 1}^{{d_1}(H)} {{i^2n_i}} &=& pq+(p+2)^2n^{'}_{p+2},
\end{eqnarray}
where $n'_{p+2}$ is the number of vertices of degree $p+2$ in $G$. \smallskip
Clearly, $n(G)=n=q(p+1)$, $m(G)=q(p+1)$, $n^{'}_{p+2}=q$. By adding (1), (2), and (3) with coefficients $2, -3, 1$, respectively, we get:
\begin{eqnarray}
\sum_{i = 1}^{{d_1}(H)} {(i^2-3i+2)n_i} = pq(p-1).
\end{eqnarray}

By Lemma \ref{lem 3-1}, $p+3\leq \mu_1(H)\leq p+3+\dfrac{2}{p+2}$. It follows from Theorem \ref{thm 2-4} that $d_1(H) +1 \leq \mu_1(H) = \mu_1(G)\leq p+3+\dfrac{2}{p+2}$, which leads to $d_1(H)\leq p+2$. Obviously, $H$ is an unicyclic graph. Since the number of spanning trees of $H$ and $G$ are the same, it is easy to see that the length of cycle of $H$ is also $q$, which implies that $H$ is also a bipartite unicyclic connected graph (Note that the number
of spanning trees of a unicyclic graph equals the length of the cycle contained in
it and the number of spanning trees in a connected graph $G$ and $H$ is $\dfrac{1}{n}\prod\limits_{i = 1}^{n - 1} {{\mu _i}}$). It follows from Lemma \ref{lem 2-13} (1) and Theorem \ref{thm 2-3}  that

\begin{eqnarray}
t(\mathcal{L}(H))=t(\mathcal{L}(G))=\sum\limits_{i = 1}^{{d_1}(H)} {{n_i}\left( {\begin{array}{*{20}{c}}
i\\
3
\end{array}} \right)}  = q\left( {\begin{array}{*{20}{c}}
p+2\\
3
\end{array}} \right).
\end{eqnarray}

We claim that $d_1(H)=p+2$. We assume on the  contrary that $d_1(H)\leq p+1$. By (5),
$q{p+2 \choose 3} = \sum_{i = 1}^{{d_1}(H)}n_i {i \choose 3} \leq \dfrac{p+1}{6}\sum\limits_{i = 1}^{{d_1}(H)} {(i^2-3i+2)n_i},$ which implies that $q{p+2 \choose 3} \leq \dfrac{p+1}{6}pq(p-1)$. This  yields that ${p + 2 \choose 3} \leq \dfrac{(p-1)p(p+1)}{6}$ $=$ ${p+1 \choose 3}$, a contradiction. By a similar argument and this fact that $d_2(H)\leq \cdots \leq d_q(H)\leq p+2$, one can easily see that $d_2(H)=d_3(H)=\cdots=d_q(H)=p+2$. On the other hand, since $\delta(H)=d_{q(p+1)}(H)\geq 1$, it follows from (2) that $d_{q+1}(H)=\cdots=d_{q(p+1)}(H)=1$ and so $deg(H)=deg(G)$, proving the lemma.
\end{proof}

\begin{theorem} \label{T1}
Let $q$ be an even number. If $H$ is $L$-cospectral to a jellyfish graph, then $H$ DLS.
\end{theorem}

\begin{proof} Let $H$ be $L$-cospectral with the jellyfish graph $G=JFG(p, q)$. It follows from Lemma \ref{lem 4-2} that $deg(G)=deg(H)$. Since $H$ is an unicyclic graph,  $H=G$.
\end{proof}
The following corollary immediately follows from Theorems  \ref{thm 2-2}  and \ref{T1}.
\begin{corollary}
Let $q$ be an even number. If $H$ is $L$-cospectral to a jellyfish graph. Then the complement of $H$ is also DLS.
\end{corollary}

In the following lemma the signless Laplacian spectrum of a graph $Q$-cospectral with a jellyfish graph is calculated.

\begin{lemma}\label{lem 5-1}
If $H$ is a graph that is $Q$-cospectral with the jellyfish graph $G = JFG(p, q)$, then the signless Laplacian spectrum of $H$ are:

$\dfrac{\lambda_i+p+3\pm \sqrt{\lambda_i^2+(2p+2)\lambda_i+p^2+2p+5}}{2}$ and 1 in which the multiplicity of 1 is an integer $a$ such that $1\leq a\leq n-q$. Here, $\lambda_i=Cos\dfrac{2\pi i}{q}$, for $i=1, 2,\cdots, q$.
\end{lemma}

\begin{proof}  By a suitable  labeling of vertices of $G = JFG(p, q)$, we may assume that
$Q(G) =\left[ {\begin{array}{*{20}{c}}
{{A_{q \times q}}}&{{B_{q \times (n - q)}}}\\
{{C_{(n - q) \times q}}}&{{D_{(n - q) \times (n - q)}}}
\end{array}} \right]$, where 
\begin{eqnarray*}
A_{q \times q} &=& (p+2){I_q} + A_{C_q}, \ B_{q \times (n-q)} = \left[ {\begin{array}{*{20}{c}}
{{I_q}}& \cdots &{{I_q}}
\end{array}} \right],\\
\hspace{30mm}C_{(n - q) \times q } &=& \left[ {\begin{array}{*{20}{c}}
{{I_q}}\\
 \vdots \\
{{I_q}}
\end{array}} \right], \ D_{(n - q) \times (n - q)}=I_{n-q}.
\end{eqnarray*}

Therefore $P_{Q(G)}(x) =det(xI_{n}-Q(G))= (x-1)^{n-q}P_{A_{C_q}}(x-(p+2)-\dfrac{1}{x-1})$, where $P_{Q(G)}(x)$ and $P_{A_{C_q}}(x)$ are characteristic polynomials of matrices $Q(G)$ and $A_{C_q}$, respectively.  It follows from Lemma \ref{lem 2-12} ($iii$) that  for $p\geq 2$, $ G$ has 1 as its eigenvalue. Hence, for $x\neq 1$, $P_{A_{C_q}}(x-(p+2)-\dfrac{1}{x-1})=0$ if and only if $P_{Q(G)}(x)=0$. Therefore, $x=\dfrac{\lambda_i+p+3\pm \sqrt{\lambda_i^2+(2p+2)\lambda_i+p^2+2p+5}}{2}$, where $\lambda_i=Cos\dfrac{2\pi i}{q}$, for $i=1, 2, \cdots, q$.
\end{proof}

\begin{corollary}\label{cor 5-2}  If $H$ is a graph that is $Q$-cospectral with $G = JFG(p, q)$, then $q_1(H)=\dfrac{p+5+\sqrt{p^2+6p+13}}{2}$.
\end{corollary}

\begin{lemma}\label{lem 5-2} If $H$ is $Q$-cospectral with $G = JFG(p, q)$, then $det(H)\in  \left\{ {0,4} \right\}$.
\end{lemma}

\begin{proof} Suppose $q\geq 4$ is an even number.   Since $G = JFG(p, q)$ is a bipartite graph, by Lemma \ref{lem 2-12} ($i$), we have $det(Q(G))=det(Q(H))=0$. If $q$ is an odd number, then $G = JFG(p, q)$ is not a  bipartite graph and so by Lemma \ref{lem 2-12} ($ii$) $det(Q(G))=det(Q(H)=4$.
\end{proof}

We are now ready to prove that a graph $Q$-cospectral with $JFG(p, q)$ have the same degree sequence as $JFG(p, q)$.

\begin{lemma}\label{lem 5-4}
If $H$ is $Q$-cospectral with $G =  JFG(p, q)$, then they have the same degree sequence.
\end{lemma}

\begin{proof} Since $H$ and $G$ are $Q$-cospectral, by Lemma  \ref{lem 2-10} and the main properties of Laplacian spectrum, $H$  has the same order, size, and first Zagreb index as $G$. Let $n_i$ denote the number of vertices of degree $i$ in $H$, $0 \leq i \leq d_1(H)$. Then,

\begin{eqnarray}
\sum_{i =0}^{{d_1}(H)} {{n_i}} &=& n(G), \\
\sum_{i = 0}^{{d_1}(H)} {{in_i}} &=& 2m(G),\\
\sum_{i = 0}^{{d_1}(H)} {{i^2n_i}} &=&pq+(p+2)^2n^{'}_{p+2},
\end{eqnarray}
where $n'_{p+2}$ is the number of vertices of degree $p+2$ in $G$.

\hspace{10mm}
\begin{figure}[h]
\centerline{\includegraphics[height=8cm]{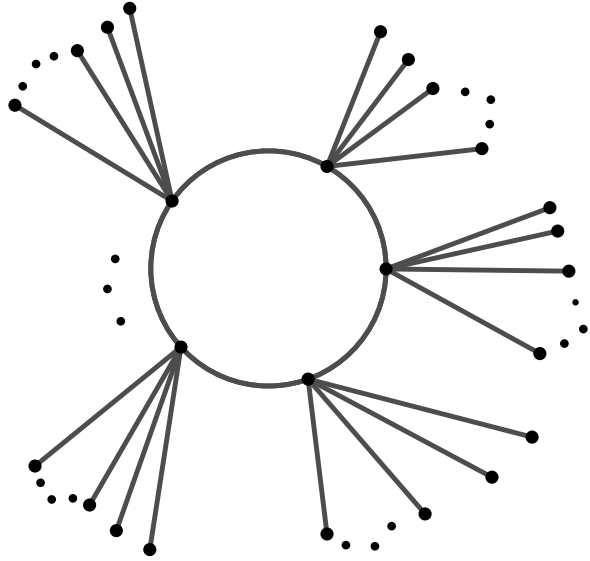}}
\begin{center}
{\caption{{The jellyfish graph $JFG(p, q)$}}
}\end{center}
\end{figure}

It is clear that $n(G)=n=q(p+1)$, $m(G)=q(p+1)$ and $n^{'}_{p+2}=q$. By adding (6), (7), and (8) with coefficients $2, -3, 1$, respectively, we get:
\begin{eqnarray}
\hspace{10mm}\sum_{i = 0}^{{d_1}(H)} {(i^2-3i+2)n_i} = pq(p-1).
\end{eqnarray}

\begin{eqnarray}
\sum_{i = 0}^{{d_1}(H)} {(i-1)n_i} =q(p+1).
\end{eqnarray}

 It follows from Lemma \ref{lem 2-13} (2) and Theorem \ref{thm 2-3}  that

\begin{eqnarray}
t(\mathcal{L}(H)) &=& t(\mathcal{L}(G))=\sum_{i = 0}^{{d_1}(H)}{{n_i} {i \choose 3}}  = q{p+2 \choose 3},\\
2m(\mathcal{L}(H)) &=& 2m(\mathcal{L}(G))=\sum\limits_{i = 0}^{{d_1}(H)} 2{{n_i}{i \choose 2}}  = 2q{p+2 \choose 2}.
\end{eqnarray}

We claim that $d_1(H)=p+2$. Suppose on the contrary that $d_1(H)\leq p+1$ or $d_1(H)\geq p+3$. We consider the following two cases:

\begin{enumerate}
\item  $d_1(H)\leq p+1$.  Then by (11), $q{p + 2 \choose 3} \leq \dfrac{p+1}{6}\sum_{i = 0}^{{d_1}(H)} {(i^2-3i+2)n_i}$. Hence $q{p + 2 \choose 3} \leq \dfrac{p+1}{6}pq(p-1)$ which yields that ${p + 2 \choose 3} \leq \dfrac{(p-1)p(p+1)}{6}$ $=$ ${p + 1 \choose 3}$, a contradiction.

\item  $d_1(H)\geq p+3$.  Then by (12), $2q{p + 2 \choose 2} \geq (p+3) \sum_{i = 0}^{{d_1}(H)} {(i-1)n_i}$. Thus, $2q{p + 2 \choose 2} \geq q(p+3)(p+1)$, which proves that $p+2\geq p+3$ that is impossible.
\end{enumerate}

 Therefore, $d_1(H)=p+2$. Since $d_2(H)\leq \cdots \leq d_q(H)\leq p+2$, by a similar argument one can see that  $d_2(H)=d_3(H)=\cdots=d_q(H)=p+2$. On the other hand, $\delta(H)=d_{q(p+1)}(H)\in  \left\{ {0,1} \right\}$. Note that $H$ has at most an isolated vertex; that is, $n_0\in  \left\{ {0,1} \right\}$. This depends on $q$ is either an even or an  odd number. Hence $d_{q(p+1)-1}(H)\geq 1$ and so  it is easy to check that $deg(H)\in  \left\{ {0,1,2,p+2} \right\}$. By (6), (7) and (8),  we get $n_0 + n_1 + n_2 = qp,$ $n_1 + 2n_2+ (p+2)q = 2q(p+1)$ and $n_1 + 4n_2+ (p+2)^2q = pq + (p+2)^2q$. This implies that $n_1=pq$ and so $n_0=n_2=0$. Therefore, $deg(H)=deg(G)$.
\end{proof}

\begin{lemma}\label{lem 5-5}
 If $H$ is $Q$-cospectral with $G = JFG(p, q)$ then $H$ is a connected graph.
\end{lemma}

\begin{proof} Suppose on contrary that $H$ is a graph with exactly $i$ connected components $H_1, H_2, \cdots,  H_i$.  We also assume that $n(H_j)=n_j$, $1 \leq j \leq i$.  Since $G$ is unicyclic, it has at most one zero eigenvalue and so one can deduce that one of the following happens:
 \begin{enumerate}
  \item $q$ is an odd number. In this case, all connected components are $k$-cyclic graphs such that at least one of these  cycles is odd. Therefore,
\begin{eqnarray*}
n&=&m(H)=m(H_1)+m(H_2)+\cdots+m(H_i)\\ &=& (n_1+k_1-1)+ (n_2+k_2-1)+\cdots+ (n_i+k_i-1)\\ &=& (n_1+n_2+\cdots+n_i)-i+(k_1+k_2+\cdots+k_i)\\ &=&n-i+(k_1+k_2+\cdots+k_i)
\end{eqnarray*}
which shows that $k_1+k_2+\cdots+k_i=i$. Since $k_i\geq 1$, $k_1=k_2=\cdots=k_i=1$. We now apply Lemma \ref{lem 2-12} (ii) to deduce that $det(Q(H))\geq 16$, contradiction to Lemma \ref{lem 5-2}.

\item $q$ is an even number. This means that $G$ is a bipartite graph.  Without loss of generality we assume  that $H_1$ is a bipartite graph and the other component are $k$-cyclic graphs such that at least one of these $k$ cycles is an odd cycle.  Consider the following subcases:
\begin{enumerate}
  \item $k_1=0$. This means that there exists $2\leq j\leq k$ such that $k_j=2$ and for any $i\neq j, 1$, $k_i=1$. By Lemma \ref{lem 5-4}, there exists a subgraph $G_1$ of $H_j$ such that $G_1 \cong JFG(p, q^{'})$.  By Corollary \ref{cor 5-2}, $q_1(G)=q_1(G_1)=\dfrac{p+5+\sqrt{p^2+6p+13}}{2}$. On the other hand, it follows from Lemma \ref{lem 2-16} that $q_1(H_j)>q_1(G_1)$. Thus, $q_1(H_j)>q_1(G)$ which is impossible.

  \item $k_1=1$. Therefore, all $H_j$ are unicyclic graphs. On the other hand, by Lemma \ref{lem 5-4}, each $H_i$ is a  jellyfish graph. By the Perron–Frobenius theorem, the multiplicity $q_1(G)=\dfrac{p+5+\sqrt{p^2+6p+13}}{2}$ is 1. On the other hand, by Corollary \ref{cor 5-2}, we get $q_1(H_1)=q_1(H_2)=\cdots=q_1(H_i)=\dfrac{p+5+\sqrt{p^2+6p+13}}{2}$. This means that the multiplicity of $q_1(G)$ is $i\geq 2$, which is our final contradiction.
  \end{enumerate}
  \end{enumerate}
Hence the result.
\end{proof}

\begin{theorem} \label{T2}
Any jellyfish graph is DQS.
\end{theorem}

\begin{proof}
Let $H$ be $Q$-cospectral with the jellyfish graph $G=JFG(p, q)$. It follows from Lemma \ref{lem 5-4}  that $deg(G)=deg(H)$. On the other hand, $H$ is an unicyclic  graph and so $H=G$.
\end{proof}

Note that the main result of this paper is a combination of Theorem \ref{T1} and Theorem \ref{T2}.

\section*{Concluding Remarks}
In this paper, it is proved that jellyfish graphs $G=JFG(p, q)$ are DQS. 	Additionally, we prove that  if $q$ is an even number, then for any    graph $H$,  $L$-cospectral to a jellyfish graph $G=JFG(p, q)$,  $H$ and its complement  are DLS. Now, we pose the following open problem.

\textbf{Conjecture}. If $q$ is an odd number and  $H$ is a  graph $L$-cospectral to a jellyfish graph $G=JFG(p, q)$, then $H$ is DLS.\\

{\bf Acknowledgement}. The research of the second author is partially supported by
the university of Kashan under grant number 890190/6.

\end{document}